\documentclass[1p]{elsarticle}

\usepackage{amsmath,amssymb,amsthm,mathtools}
\usepackage{microtype}
\usepackage{bm}

\journal{Applied Mathematics Letters}

\newtheorem{proposition}{Proposition}
\newtheorem{theorem}{Theorem}

\newcommand{\ii}{\mathrm i}
\newcommand{\cT}{\mathcal T}
\newcommand{\cF}{\mathcal F}
\newcommand{\cX}{\mathcal X}
\newcommand{\avg}[1]{\{\!\{#1\}\!\}}
\newcommand{\jump}[1]{[\![#1]\!]}
\newcommand{\TDG}{\mathrm{TDG}}

\begin{document}
\begin{frontmatter}

\title{The PWDG--UWVF Equivalence Revisited}

\author[1]{Shelvean Kapita}
\address[1]{Department of Mathematics, Texas A\&M University, College Station, TX, USA}

\begin{abstract}
The ultra-weak variational formulation (UWVF) is known to be the Trefftz discontinuous Galerkin scheme obtained from the standard PWDG flux family by setting \(\alpha=\beta=\delta=1/2\). We give a short structural explanation of this fact. On every interior mesh edge, the two PWDG jump terms at these parameter values are exactly one half of the sum of the squared norms of the two oriented impedance-transmission defects used by the UWVF. Thus the Trefftz-DG norm is the natural Riesz norm of the global transmission defect. In PWDG coordinates its Gramian is the anti-Hermitian part of the discrete matrix, and Riesz normalization puts the operator in the normal form \(K+\ii I\), with \(K\) Hermitian. The same argument applied directly to the discrete UWVF yields the complementary normal form \(\tfrac12 I+\ii K_{\rm U}\). This identifies the classical equivalence as a relation between broken-field and impedance-trace coordinates for the same edge-coupling problem, rather than only as a particular choice of penalties.
\end{abstract}


\end{frontmatter}

\section{Introduction}

The ultra-weak variational formulation (UWVF) of Cessenat and Despr\'es uses impedance traces on element boundaries as unknowns \cite{CessenatDespres1998,CessenatDespres2003}. Plane-wave discontinuous Galerkin (PWDG), or more generally Trefftz-DG, instead uses broken Helmholtz fields and couples neighboring elements by numerical fluxes. The two methods have long been known to coincide for the particular choice
\begin{equation}\label{eq:half}
        \alpha=\beta=\delta=\frac12;
\end{equation}
see \cite[Remark 2.1]{GittelsonHiptmairPerugia2009} and \cite[Sec. 2.2]{HiptmairMoiolaPerugia2016}. The usual derivation verifies this by inserting the UWVF fluxes into the DG formulation. It does not, however, make transparent why the factors \(1/2\) arise.

The point of this note is that \eqref{eq:half} is forced by a simple impedance identity. The Dirichlet and normal-flux jumps used by PWDG are the two real channels of the complex transmission residual used by the UWVF. At the values \eqref{eq:half}, their weighted squares are exactly the norm of the two oriented impedance defects on each mesh edge. This also gives a direct Riesz interpretation of the Trefftz-DG norm. Since
\begin{equation}\label{eq:coercivity-intro}
        \operatorname{Im} A_h(v,v)=\lVert v\rVert_{\TDG}^2
\end{equation}
for Trefftz fields \cite{HiptmairMoiolaPerugia2016}, the corresponding Gram matrix is the anti-Hermitian part of the PWDG matrix. Its inverse square root therefore produces a normal form canonical relative to the TDG Riesz metric, rather than merely rescaling an ill-conditioned plane-wave basis. The stability identity itself and the ensuing matrix algebra are standard; the new observations here are the edgewise impedance isometry explaining \eqref{eq:half} and the paired PWDG--UWVF Riesz normal forms.

For direct comparison with the classical UWVF literature we use the sign convention \(\partial_nu+\ii ku=g\). Conjugating all impedance traces gives the corresponding formulas for the opposite time convention.

\section{The two formulations}

Let \(\Omega\subset\mathbb R^2\) be polygonal and let \(\cT_h\) be a mesh with interior edges \(\cF_h^I\) and boundary \(\Gamma=\partial\Omega\). We consider
\begin{equation}\label{eq:model}
 -\Delta u-k^2u=0\quad\hbox{in }\Omega,
 \qquad \partial_nu+\ii ku=g\quad\hbox{on }\Gamma.
\end{equation}
For an interior edge \(e=\partial K^+\cap\partial K^-\), with outward normals \(n^\pm\), set
\[
 \jump{v}_N=v^+n^++v^-n^-,\qquad
 \jump{\nabla v}_N=\nabla v^+\!\cdot n^++\nabla v^-\!\cdot n^-,
\]
and use the usual arithmetic averages \(\avg{v}\) and \(\avg{\nabla v}\). Let \(V_h\) be any finite-dimensional broken Trefftz space, so \(-\Delta v-k^2v=0\) on every element.

\subsection{PWDG}
The standard primal Trefftz-DG formulation \cite{GittelsonHiptmairPerugia2009,HiptmairMoiolaPerugia2016} is: find \(u_h\in V_h\) such that
\begin{equation}\label{eq:pwdg}
        A_h(u_h,v_h)=\ell_h(v_h)\qquad\forall v_h\in V_h,
\end{equation}
where, for \(\alpha,\beta>0\) and \(0<\delta<1\),
\begin{align}
A_h(u,v)={}&\int_{\cF_h^I}\!\Big(
 \avg{u}\,\jump{\nabla\bar v}_N
 -\avg{\nabla u}\!\cdot\jump{\bar v}_N
 +\alpha\ii k\,\jump{u}_N\!\cdot\jump{\bar v}_N \notag\\[-1mm]
&\hspace{24mm}-\beta(\ii k)^{-1}\jump{\nabla u}_N\jump{\nabla\bar v}_N\Big)\,ds \notag\\
&+\int_\Gamma\!\Big((1-\delta)\ii k u\bar v+(1-\delta)u\,\overline{\partial_nv}
 -\delta\,\partial_nu\,\bar v \notag\\[-1mm]
&\hspace{29mm}-\delta(\ii k)^{-1}\partial_nu\,\overline{\partial_nv}\Big)\,ds, \label{eq:Ah}
\end{align}
and
\begin{equation}\label{eq:ell}
 \ell_h(v)=\int_\Gamma g\Big((1-\delta)\bar v-\delta(\ii k)^{-1}\overline{\partial_nv}\Big)\,ds.
\end{equation}
The associated Trefftz-DG norm is
\begin{align}
\lVert v\rVert_{\TDG}^2={}&k^{-1}\lVert\beta^{1/2}\jump{\nabla v}_N\rVert_{\cF_h^I}^2
+k\lVert\alpha^{1/2}\jump v_N\rVert_{\cF_h^I}^2 \notag\\
&+k^{-1}\lVert\delta^{1/2}\partial_nv\rVert_\Gamma^2
+k\lVert(1-\delta)^{1/2}v\rVert_\Gamma^2. \label{eq:tdgnorm}
\end{align}
For Trefftz functions, \eqref{eq:coercivity-intro} follows directly from \eqref{eq:Ah}. For the impedance problem \eqref{eq:model}, \eqref{eq:tdgnorm} is a norm on the discrete Trefftz space: zero jumps and zero boundary contribution give a global homogeneous solution, hence zero by uniqueness.

\subsection{UWVF}
Let
\[
        \cX=\prod_{K\in\cT_h}L^2(\partial K).
\]
For \(y_K\in L^2(\partial K)\), let \(e_K\) solve the local adjoint-impedance problem
\[
 -\Delta e_K-k^2e_K=0\ \hbox{in }K,
 \qquad (-\partial_{n_K}+\ii k)e_K=y_K\ \hbox{on }\partial K,
\]
and define the local scattering map
\begin{equation}\label{eq:FK}
        F_Ky_K=(\partial_{n_K}+\ii k)e_K.
\end{equation}
For the unit-impedance boundary condition in \eqref{eq:model}, the UWVF reads \cite{CessenatDespres1998,HiptmairMoiolaPerugia2016}: find \(x\in\cX_h\subset\cX\) such that, for all \(y\in\cX_h\),
\begin{align}
&\sum_{K\in\cT_h}\int_{\partial K}x_K\bar y_K\,ds
-\sum_{K\in\cT_h}\sum_{K'\sim K}
  \int_{\partial K\cap\partial K'}x_{K'}\,\overline{F_Ky_K}\,ds \notag\\
&\hspace{36mm}=\sum_{K\in\cT_h}\int_{\partial K\cap\Gamma}g\,\overline{F_Ky_K}\,ds. \label{eq:uwvf}
\end{align}
Here \(K'\sim K\) denotes the element across an interior edge and
\[
 \cX_h=\{x\in\cX:x_K=(-\partial_{n_K}+\ii k)v_h|_K,\ v_h\in V_h\}.
\]
The local field is recovered from its incoming trace \(x_K\); \(F_Kx_K\) is its outgoing trace. Gittelson, Hiptmair and Perugia showed that \eqref{eq:pwdg}--\eqref{eq:ell} with \(\alpha=\beta=\delta=1/2\) is precisely \eqref{eq:uwvf} written in broken-field variables \cite[Remark 2.1]{GittelsonHiptmairPerugia2009}.

\section{Why the parameters are one half}

Consider one interior edge \(e=\partial K\cap\partial K'\), and abbreviate
\[
 u_K=u|_K,\quad u_{K'}=u|_{K'},\quad
 q_K=\partial_{n_K}u_K,\quad q_{K'}=\partial_{n_{K'}}u_{K'}.
\]
Introduce the normalized incoming and outgoing traces
\begin{equation}\label{eq:gamma}
 \gamma_K^-u=\frac{-q_K+\ii ku_K}{\sqrt{2k}},\qquad
 \gamma_K^+u=\frac{ q_K+\ii ku_K}{\sqrt{2k}}.
\end{equation}
Exact transmission is \(\gamma_K^-u=\gamma_{K'}^+u\) and \(\gamma_{K'}^-u=\gamma_K^+u\). Define the two oriented defects
\[
 r_K=\gamma_K^-u-\gamma_{K'}^+u,
 \qquad r_{K'}=\gamma_{K'}^-u-\gamma_K^+u.
\]

\begin{proposition}[Impedance identity]\label{prop:identity}
For every broken field with well-defined traces on \(e\),
\begin{equation}\label{eq:identity}
 \frac12\bigl(\lVert r_K\rVert_e^2+\lVert r_{K'}\rVert_e^2\bigr)
 =\frac1{2k}\lVert q_K+q_{K'}\rVert_e^2
  +\frac{k}{2}\lVert u_K-u_{K'}\rVert_e^2.
\end{equation}
Thus the interior part of \eqref{eq:tdgnorm} for \(\alpha=\beta=1/2\) is exactly one half of the squared UWVF transmission defect.
\end{proposition}

\begin{proof}
From \eqref{eq:gamma},
\[
 r_K=\frac{-(q_K+q_{K'})+\ii k(u_K-u_{K'})}{\sqrt{2k}},\quad
 r_{K'}=\frac{-(q_K+q_{K'})-\ii k(u_K-u_{K'})}{\sqrt{2k}}.
\]
Use \(|a+\ii b|^2+|a-\ii b|^2=2|a|^2+2|b|^2\), valid for complex \(a,b\), and integrate over \(e\).
\end{proof}

The boundary factor is equally explicit. For either element-boundary trace in \eqref{eq:gamma},
\begin{equation}\label{eq:boundary-half}
 \frac12\bigl(\lVert\gamma^-v\rVert_\Gamma^2+\lVert\gamma^+v\rVert_\Gamma^2\bigr)
 =\frac1{2k}\lVert\partial_nv\rVert_\Gamma^2+\frac{k}{2}\lVert v\rVert_\Gamma^2,
\end{equation}
which is exactly the boundary contribution in \eqref{eq:tdgnorm} for $\delta=1/2$. More generally, replacing \eqref{eq:gamma} by $-\sqrt{\beta/k}\,q_K+\ii\sqrt{\alpha k}\,u_K$ and its outgoing counterpart produces the jump weights $\beta/k$ and $\alpha k$; thus $\alpha/\beta$ selects the effective impedance $k\sqrt{\alpha/\beta}$.

\section{Riesz normal forms}

Let \(\{\Phi_j\}_{j=1}^N\) be a basis of \(V_h\), and define the PWDG matrix and DG Gram matrix by
\[
 A_{j\ell}=A_h(\Phi_\ell,\Phi_j),\qquad
 G_{j\ell}=(\Phi_\ell,\Phi_j)_{\TDG}.
\]
Thus, for \(v=\sum_jc_j\Phi_j\), \(c^*Ac=A_h(v,v)\) and \(c^*Gc=\lVert v\rVert_{\TDG}^2\).

\begin{theorem}[PWDG Riesz normal form]\label{thm:normal}
Assume that \(G\) is positive definite on the discrete Trefftz space. Then
\begin{equation}\label{eq:Ganti}
        G=\frac{A-A^*}{2\ii}.
\end{equation}
With \(B=G^{-1/2}\),
\begin{equation}\label{eq:normal}
        B^*AB=K+\ii I,\qquad K=K^*.
\end{equation}
Hence the Riesz-normalized PWDG matrix is normal and its spectrum lies on \(\operatorname{Im}z=1\).
\end{theorem}

\begin{proof}
For the sign convention in \eqref{eq:model}, elementwise Green's identity gives \(\int_{\partial K}v\,\overline{\partial_{n_K}v}\,ds=\int_K(|\nabla v|^2-k^2|v|^2)\,dx\in\mathbb R\); hence \(\operatorname{Im}A_h(v,v)=\lVert v\rVert_{\TDG}^2\) with the positive sign used here. Equation \eqref{eq:Ganti} then follows by polarization. Writing \(H=(A+A^*)/2\) gives \(A=H+\ii G\), so
\[
 B^*AB=B^*HB+\ii B^*GB=K+\ii I,
\]
with \(K=B^*HB\) Hermitian.
\end{proof}

The normal form is basis independent in the following precise sense. If \(Hx=\mu Gx\), then the \(\mu\)'s are unchanged by any nonsingular change of Trefftz basis, while the normalized eigenvalues are \(\mu+\ii\) and the singular values are \(\sqrt{1+\mu^2}\). Thus the raw \(\kappa_2(A)\), which can vary drastically under rescaling or a nearly redundant plane-wave dictionary, is separated from the intrinsic conditioning of the variational operator in its DG Riesz metric. If tiny eigenvalues of \(G\) are truncated before applying \(G^{-1/2}\), the same normal form holds on the retained numerical range.

The corresponding statement for the discrete UWVF is complementary and does not require exact invariance of the trace space under the continuous scattering maps. After orthonormalizing the discrete impedance-trace metric, write
\begin{equation}\label{eq:AU}
        A_{\rm U}=I-U_h,
\end{equation}
where $U_h$ is the discrete propagation-and-edge-exchange operator. For the absorbing problem \eqref{eq:model}, and also for an orthogonal compression of a lossless scattering operator, $U_h$ is contractive. Therefore
\begin{equation}\label{eq:GU}
 G_{\rm U}:=A_{\rm U}+A_{\rm U}^*
 =(I-U_h)^*(I-U_h)+(I-U_h^*U_h)\ge0.
\end{equation}
Whenever $G_{\rm U}>0$, $B_{\rm U}=G_{\rm U}^{-1/2}$ gives
\begin{equation}\label{eq:Uform}
 B_{\rm U}^*A_{\rm U}B_{\rm U}
 =\frac12I+\ii K_{\rm U},\qquad K_{\rm U}=K_{\rm U}^*.
\end{equation}
The first term in \eqref{eq:GU} measures transmission mismatch and the second measures dissipation caused by the absorbing boundary or by discrete compression. In the lossless invariant case $U_h$ is unitary, the second term vanishes, and $U_hq=e^{\ii\theta}q$ implies $G_{\rm U}q=4\sin^2(\theta/2)q$.

\section{A small numerical check}

We use the unit square divided into two triangles, $k=4$, and $p$ equispaced plane waves on each element, centered at its centroid. All edge integrals use 80-point Gauss--Legendre quadrature. The PWDG matrix is assembled from \eqref{eq:Ah} with $\alpha=\beta=\delta=1/2$. Independently, the UWVF matrix is assembled from the normalized traces \eqref{eq:gamma}: the diagonal blocks are incoming-trace Gram matrices, while an interior-edge block pairs the neighboring incoming trace with the outgoing test trace. Set
\[
 \widehat A_{\rm P}=G_{\rm P}^{-1/2}A_{\rm P}G_{\rm P}^{-1/2},\qquad
 \widehat A_{\rm U}=G_{\rm U}^{-1/2}A_{\rm U}G_{\rm U}^{-1/2}.
\]
\begin{table}[h]
\centering
\small
\caption{Raw and Riesz-normalized condition numbers.}
\label{tab:conds}
\begin{tabular}{rcccc}
\hline
$p$ & $\kappa_2(A_{\rm P})$ & $\kappa_2(\widehat A_{\rm P})$ & $\kappa_2(A_{\rm U})$ & $\kappa_2(\widehat A_{\rm U})$\\
\hline
5  & $4.18$              & $1.192$ & $4.18$              & $1.192$\\
9  & $1.81\!\times\!10^{2}$ & $1.374$ & $1.81\!\times\!10^{2}$ & $1.374$\\
13 & $9.12\!\times\!10^{4}$ & $1.755$ & $9.12\!\times\!10^{4}$ & $1.755$\\
17 & $3.05\!\times\!10^{8}$ & $2.379$ & $3.05\!\times\!10^{8}$ & $2.379$\\
\hline
\end{tabular}
\end{table}
With the normalization \eqref{eq:gamma}, the independently assembled matrices satisfy
\[
 \frac{\|A_{\rm U}+\ii A_{\rm P}\|_F}{\|A_{\rm U}\|_F}\le4.5\times10^{-16},\qquad
 \frac{\|G_{\rm U}-2G_{\rm P}\|_F}{\|G_{\rm U}\|_F}\le4.2\times10^{-16}
\]
over all four rows. Thus the computation verifies the equivalence at $1/2$ to roundoff, while the raw condition number grows by eight orders of magnitude and the Riesz-normalized condition number remains below $2.4$.

\section{Discussion}

Proposition~\ref{prop:identity} and \eqref{eq:boundary-half} explain all three classical values $\alpha=\beta=\delta=1/2$: the separate value and normal-flux channels of PWDG become, up to the fixed factor $1/2$, an isometric representation of the oriented impedance defects used by UWVF. The methods therefore describe the same edge-coupling problem in broken-field and impedance-trace coordinates.

The Riesz normal forms sharpen this interpretation. In PWDG coordinates the skew-Hermitian part is $\ii G_{\rm P}$ and gives $K+\ii I$; in discrete UWVF coordinates the Hermitian defect metric $G_{\rm U}=A_{\rm U}+A_{\rm U}^*$ gives $\tfrac12I+\ii K_{\rm U}$. This is related to operator preconditioning \cite{Hiptmair2006} and local UWVF boundary orthonormalization \cite{HuttunenMonkKaipio2002}, but the PWDG Riesz matrix here is assembled directly from the jump and boundary terms already present in the TDG stability norm.

Under a nonsingular Trefftz basis change $S$, $A\mapsto S^*AS$ and $G\mapsto S^*GS$, so raw matrix conditioning depends on the representation while the generalized pencil $(H,G)$ and the Riesz-normalized spectrum do not. If $G=V\Lambda V^*$ has tiny eigenvalues, retaining only $V_r$ with $\lambda_j\ge\tau\lambda_{\max}$ and using $B_\tau=V_r\Lambda_r^{-1/2}$ gives the same normal form for the compressed pencil on the retained subspace. This separates nearly invisible Trefftz combinations from intrinsic conditioning of the discrete operator.

\end{document}